\documentclass[oneside,12pt,reqno]{amsart}
\usepackage{enumitem} 

\usepackage{amssymb}
\usepackage{amsmath,amsthm,amscd}
\usepackage{verbatim}
\usepackage{graphicx,hyperref}
\usepackage[all]{xy}

\usepackage[usenames,dvipsnames]{xcolor}
\usepackage{amsthm, amsfonts, graphicx, pinlabel, tikz}
\usepackage[all]{xy}
\usepackage{hyperref} 
\usepackage{mathtools}

\usepackage{amssymb}
\usepackage{amsmath,amsthm,amscd}
\usepackage{verbatim}
\usepackage{graphicx,hyperref}

\newcommand{\RR}{\mathbf{R}}

\newcommand{\C}{\mathbb{C}}
\newcommand{\CP}{\mathbb{C}P}

\newcommand{\R}{\mathbb{R}}

\newcommand{\Z}{\mathbb{Z}}

\newcommand{\id}{\mathrm{Id}}

\newcommand{\OP}{\operatorname}

\renewcommand{\Re}[1]{\mathfrak{Re}\,#1}

\renewcommand{\Re}{\mathfrak{Re}}

\theoremstyle{plain}
\newtheorem{thm}{Theorem}[section]
\newtheorem{cor}[thm]{Corollary}
\newtheorem{lem}[thm]{Lemma}

\newtheorem{prop}[thm]{Proposition}

\newtheorem{mainthm}{Theorem}
\newtheorem{maincor}[mainthm]{Corollary}

\theoremstyle{definition}
\newtheorem{defn}[thm]{Definition}

\theoremstyle{remark}
\newtheorem{rem}[thm]{Remark}
\newtheorem{ex}[thm]{Example}

\numberwithin{equation}{section}

\title[$C^0$-Legendrian knots and non-squeezing]{$C^0$-limits of Legendrian knots and contact non-squeezing}

\author{Georgios Dimitroglou Rizell}
\address{Department of Mathematics\\
Uppsala University\\
Box 480\\
SE-751 06 Uppsala\\
Sweden}
\email{georgios.dimitroglou@math.uu.se}

\author{Michael G. Sullivan}
\address{Department of Mathematics and Statistics\\
University of Massachusetts\\
Amherst\\
MA 01003\\
USA}
\email{sullivan@math.umass.edu}

\thanks{The first author is supported by the Knut and Alice Wallenberg Foundation under the grant KAW 2016.0198, and by the Swedish Research Council under the grant number 2020-04426. The second author is supported by the Simons Foundation grant number 708337.
The authors are grateful to: Paolo Ghiggini who taught the authors about properties of Legendrians in neighborhoods of transverse knots; Sobhan Seyfaddini for pointing out relevant questions and showing interest in the work; and Thomas Kragh for pointing out that Theorem \ref{thm:c0} indeed is sufficiently strong to settle Gromov's Alternative.}

\begin{document}

\begin{abstract}
Take a sequence of contactomorphisms of a contact three-manifold that $C^0$-converges to a homeomorphism. If the images of a Legendrian knot limit to a smooth knot under this sequence, we show that it is Legendrian. We prove this by establishing that, on one hand, non-Legendrian knots admit a type of contact-squeezing onto transverse knots while, on the other, Legendrian knots do not admit such a squeezing. The non-trivial input from contact topology that is needed is (a local version of) the Thurston--Bennequin inequality.
\end{abstract}

\maketitle

\section{Introduction and results}
A knot $K$ inside a contact 3-manifold $(M^3,\xi)$ is {\bf Legendrian} (resp. {\bf transverse}) if, for \emph{all} points $p \in K,$ $T_pK \subset \xi_p$ (resp. $T_p K \not \subset \xi_p$). 
In this article, all knots are considered to be smooth co-orientable embedding of $S^1$ into a contact 3-manifold, where the contact structure of the latter moreover is assumed to be co-orientable; we do not make additional assumptions on the ambient contact manifold, i.e.~it can be either closed or open.
Generalizing the notion of transverse, the knot $K$ is called {\bf non-Legendrian} if, for \emph{some} $p \in K,$  $T_p K \not \subset \xi_p.$
Both Legendrian and transverse knots have been widely studied, and each class exhibits various interesting rigidity phenomena.
Non-Legendrian knots are somewhat more flexible, especially when considered from a quantitative viewpoint; for example, in the case when there exists a contactomorphism of $(M^{2n+1}, \xi)$ that connects two non-Legendrian $n$-dimensional submanifolds, Rosen--Zhang \cite[Section 1]{RosenZhang} have shown that there exists such a contactomorphism of arbitrarily small Hofer norm.

General non-Legendrian knots in the contact geometric setting have not received the same amount of attention as transverse and Legendrian knots. This article shows that non-Legendrian knots behave more like transverse knots than Legendrian knots, at least when it comes to quantitative questions. Indeed, the starting point of the results of this article is the following type of flexibility: a non-Legendrian knot can be ``squeezed'' arbitrarily close to some given transverse knot. (See Theorem \ref{thm:squeezing} below for the precise statement.)

In the following we fix an arbitrary Riemannian metric on $M$ inducing a distance function $d,$ and denote by
$$B_r(K) \coloneqq  \{x \in M; \:\: d(K,x)<r \} \subset M$$
the set of points of distance less than $r$ from the subset $K \subset M.$ 

The results in this paper are closely connected to the concept of contact squeezing. We begin with a version connected to contact isotopies.
\begin{defn}
\label{defn:squeezingIso}
Let $K_0,K \subset (M,\xi)$ be submanifolds of a contact manifold. We say that the \emph{contact isotopy} $\varphi^t \colon M \to M$ {\bf squeezes} $K_0$ {\bf onto} $K$ if there exists $\epsilon(t)$ with $\lim_{t \to +\infty} \epsilon(t) =0$ such that for all $t \gg 0$ sufficiently large, $\varphi^t(K_0) \subset B_{\epsilon(t)}(K)$ and $\varphi^t(K_0)$ 
is smoothly isotopic to $K$ inside $B_{\epsilon(t)}(K).$
\end{defn}
One of our main results is that non-Legendrian knots are flexible in the sense that they can be squeezed onto transverse knots.
\begin{mainthm}
\label{thm:squeezing}
Let $K \subset (M^3,\xi)$ be a non-Legendrian knot. There exists a transverse knot $T \subset (M^3,\xi)$ and a contact isotopy $\varphi^t \colon M \xrightarrow{\cong} M$ that squeezes $K$ onto $T$.

In particular, by replacing $M$ with a small tubular neighborhood of $K$, we can assume that the transverse knot $T$ lives in that neighborhood.
\end{mainthm}
We show that Legendrians cannot be squeezed onto transverse knots, and therefore, by the transitivity of the squeezing property provided by Part (ii) of Lemma \ref{lem:transitivity}, they also cannot be squeezed onto non-Legendrian knots. 

\begin{ex}
\label{ex:weaksqueezing}
  In certain contact manifolds one can find a contact isotopy, a Legendrian knot $K_0=\Lambda_0$, and a transverse knot $K=T$, that satisfies Part (1) of Definition \ref{defn:squeezingIso}; this is the reason why we want to define squeezing as something stronger than merely what is postulated in Part (1). For such an example, consider the contact manifold given as the ideal boundary $\partial_\infty (\C^* \times \C) \cong S^1 \times S^2$
of the Weinstein manifold $\C^* \times \C$, and the Legendrian core given as a connected component
$$\Lambda_0 \subset \partial_\infty(S^1 \times \Re(\C)) \subset \partial_\infty (\C^* \times \C)$$
of the Legendrian link at infinity. The Legendrian $\Lambda_0$ is shown in the Kirby diagram in Figure \ref{fig:stabilized}. It is homologically essential, and Legendrian isotopic to a two-fold stabilization of itself, consisting of one positive and one negative stabilization; see e.g.~\cite[Figure 19]{DingGeiges2} for more details. Now consider the transverse core given as a connected component
$$T \subset \partial_\infty(\C^* \times \{0\}) \subset \partial_\infty (\C^* \times \C)$$
of a transverse two-component link at infinity. It is possible to $C^0$-approximate $T$ by a sufficiently stabilized Legendrian core in the same smooth isotopy class. For example, the upper figure in Figure \ref{fig:stabilized} depicts a transverse arc that is approximated by a Legendrian with many positive stabilizations. In particular, there is a Legendrian isotopy of the Legendrian core $\Lambda$ into an arbitrarily small neighborhood of the transverse core, so that the Legendrian moreover is smoothly isotopic to the transverse knot inside the same neighborhood. Note that if a Legendrian has many positive and negative stabilizations, then the negative stabilizations can be shrunk arbitrarily, in order to not interfere with the approximation that is made by using the positive stabilizations.
\end{ex}

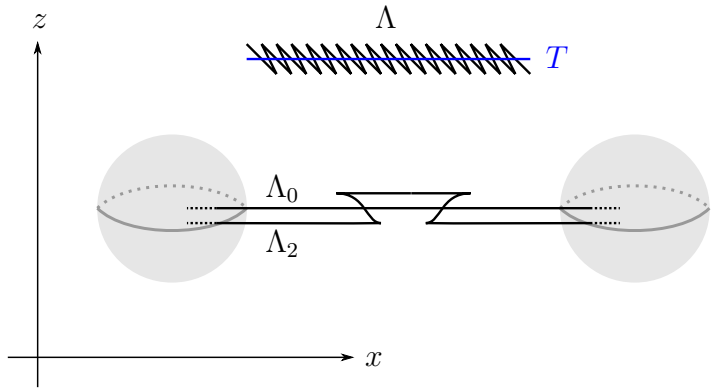
\begin{figure}[htp]
  \vspace{5mm}
	\labellist
	\pinlabel $z$ at 8 92
	\pinlabel $x$ at 92 7
        \pinlabel $\color{blue}T$ at 138 83
        \pinlabel $\Lambda$ at 95 93
        \pinlabel $\Lambda_0$ at 69 50
        \pinlabel $\Lambda_2$ at 69 36
	\endlabellist
	\includegraphics[scale=1.5]{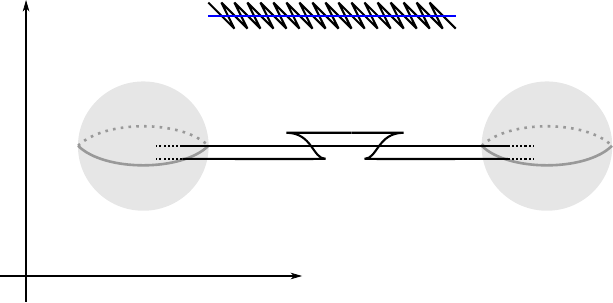}\\
	\caption{Above: a transverse knot $T=\{z=z_0,\: y=-1\}$ can be approximated by a Legendrian knot $\Lambda$ if the latter is sufficiently stabilized (stabilizations correspond to zig-zags). Below: A homologically essential Legendrian knot $\Lambda_0$ inside $\partial_\infty (\C^* \times \C)$, depicted as a Kirby diagram with a single Weinstein one-handle attached to $S^3$. The Legendrian $\Lambda_0$ is Legendrian isotopic to its two-fold stabilization $\Lambda_2$ and thus; by induction, it is Legendrian isotopic to a $2k$-fold stabilization for any $k \ge 0$. }
	\label{fig:stabilized}
\end{figure}

Non-squeezing results are a central theme in symplectic topology, going back to Gromov's famous non-squeezing result in symplectic manifolds \cite{Gromov:Pseudo}. In the field of contact topology, the notion of non-squeezing has been established for certain open subsets of certain contact manifolds by Eliashberg--Kim--Polterovich \cite{Eliashberg:Orderability}. The concept of non-squeezing in the latter article can be seen to be related to the concept studied here. In particular, note that the subsets studied there are solid tori in contact manifolds. In addition, we established a non-squeezing result for certain non-loose Legendrians onto loose Legendrians \cite[Theorem 1.7]{Dimitroglou:Persistence} . 
This result was generalized in \cite[Corollary 1.12]{Lazarev1910.01101}. The aforementioned articles established non-squeezing in arbitrary dimensions using holomorphic curve technology. The results in this article are based on parts of the theory of convex surfaces that so far only has been thoroughly developed in dimension three.

The classification of contact structure on solid tori by Giroux \cite{MR1779622} and Honda \cite{Honda:Classification}, based upon the convex surface theory by Giroux \cite{Giroux}, implies that Legendrian approximations of transverse knots must be stabilized. More precisely:
\begin{thm}[Giroux \cite{MR1779622} and Honda \cite{Honda:Classification}]
  \label{thm:girouxhonda}
For a Legendrian knot $\Lambda$ that lives inside a tubular neighborhood of a transverse knot, with the additional assumption that the two knots are smoothly isotopic inside the same neighborhood, one can give a bound from below on the number of stabilizations that the Legendrian has in terms of the distance from the Legendrian to the transverse knot. Furthermore, this number tends to $+\infty$ as this distance tends to zero.
\end{thm}
\begin{rem}When the Legendrian knot is null-homologous, and thus has a well-defined Thurston--Bennequin invariant, it immediately follows from the aforementioned result that Legendrians cannot be squeezed onto transverse knots. Section \ref{sec:TransverseNonSqueezing} is dedicated to extending this result from transverse to arbitrary non-Legendrian knots.
\end{rem}
To the authors' knowledge, Theorem \ref{thm:girouxhonda} has not been explicitly stated in the literature. Since our work here do not rely on the above result, but rather use weaker results in the same spirit that concern relative Thurston--Bennequin numbers, we only provide a brief sketch of the ideas that go into the proof.
\begin{proof}[Sketch of proof of Theorem \ref{thm:girouxhonda}]
  Consider a Legendrian $\Lambda$ which is close to a transverse knot $T$ in the same isotopy class. By Giroux' theory of convex surfaces \cite{Giroux}, one can produce an embedded convex annulus $A$ inside the normal neighborhood of the transverse knot with boundary $\partial A=\Lambda \sqcup \Lambda_k$. Here $\Lambda_k$ is the standard $k$-fold stabilized Legendrian approximation of the transverse knot $T$ described in Section \ref{sec:Lambdak}, which is contained on the boundary of a tubular neighborhood of $T$, while $\Lambda$ is contained in the interior of the neighborhood. 
  
We use the language of \cite{Honda:Classification}.  A sufficiently small tubular neighborhood of the transverse knot is tight. So the dividing curves of the convex tori inside this neighborhood satisfy the minimally twisting property.
Consider the dividing curves of the annulus $A.$
The minimally twisting property implies the existence of bypass half-disks in $A$ for the boundary component $\Lambda \subset \partial A.$
The bypass half-disks give the sought destabilizations of $\Lambda.$ 
\end{proof}
 The proof of our non-squeezing result Theorem \ref{thm:TransverseNonSqueezing} does not rely on the fact that a Legendrian that is close to a transverse knot in the same isotopy class must be stabilized; however, the proof establishes that its relative Thurston--Bennequin number admits a bound from above, where this bounds moreover tends to $-\infty$ as the distance to the transverse knot tends to zero.  If one would like to deduce the existence of stabilizations for the knot, one could subsequently use the classification result for Legendrian knots by Eliashberg--Fraser \cite{TopTrivLeg} or Ding--Geiges \cite{DingGeiges}.
 
It turns out that the only ingredient from the classification of contact structures that is needed for Theorem \ref{thm:TransverseNonSqueezing} is the \emph{Thurston--Bennequin inequality for Legendrian unknots in $\R^3$} as proven by Bennequin in \cite{Bennequin:Entrelacements}. Of course, this inequality is also highly non-trivial, as it e.g.~implies that the standard contact 3-sphere is tight. 

In order to deduce properties for $C^0$-limits of Legendrian knots, we need to consider a weaker notion of squeezing. One of the crucial results is that Legendrians also cannot be squeezed onto non-Legendrians in this weaker sense.
\begin{defn}
  \label{defn:squeezing}
We say that the \emph{sequence of contactomorphisms} $\varphi_i \colon (M,\xi) \to (M,\xi)$ {\bf squeezes} $K_0 \subset M$ {\bf onto} $K \subset M,$ where $K_0$ and $K$ are submanifolds, if the following holds.
\begin{enumerate}
\item 
There exists $\epsilon_i>0$ with $\lim_{i \to +\infty} \epsilon_{i} = 0$ such that for all $i \gg 0,$ $\varphi_{i}(K_0) \subset B_{\epsilon_i}(K)$ and $\varphi_{i}(K_0)$ is smoothly isotopic to $K$ inside $B_{\epsilon_i}(K).$  
\item 
For any $r > 0$ and $\epsilon>0$, there exists some $i_{r,\epsilon} \gg 0$ such that
  $$d(\varphi_i \circ \varphi_{j}^{-1}(x),x) < \epsilon$$
  for all $i \ge j \ge i_{r,\epsilon}$ and $x \in M \setminus B_{r}(K)$.

\end{enumerate}
\end{defn}

Part (2) of Definition \ref{defn:squeezing} is the counterpart of the second part of Definition \ref{defn:squeezingIso}. As in Example \ref{ex:weaksqueezing} one can produce a sequence of contactomorphisms for which a Legendrian knot $K_0=\Lambda_0$ and transverse knot $K=T$ satisfies Part (1) of Definition \ref{defn:squeezing}; this is the reason why we want to require something stronger for the notion of squeezing. We are not sure if Part (2) is the most natural definition if one wants a notion of squeezing that precludes the possibility of squeezing a Legendrian onto a transverse knot. However, as we prove in Section \ref{sec:TransitivityProof}, one good feature of the above definition is that the existence of squeezing sequences become transitive in the following manner.

\begin{lem}
  \label{lem:transitivity}
  \begin{enumerate}[label=(\roman*)]\item If there exists a contact isotopy $\psi_t \colon M \to M$ that squeezes a submanifold $K_0 \subset M$ onto a submanifold $K \subset M$ (see Definition \ref{defn:squeezingIso}) then one can produce a squeezing sequence $\varphi_i \colon M \xrightarrow{\cong} M$ of contactomorphisms of $K_0$ onto $K$ (see Definition \ref{defn:squeezing}). Moreover, the support of $\varphi_i$ can be assumed to be contained inside the support of $\psi_t$.
  \item Consider two sequences of contactomorphisms
    $$\varphi^{(\nu)}_i \colon M \to M, \:\: \nu =1,2,$$
    where $\{\varphi^{(\nu)}_i\}$ squeezes $K_\nu$ onto $K_{\nu-1}$. Then there exists a suitable re-indexing $\alpha(i)\ge i$ for which
    $$ \varphi^{(1)}_i \circ \varphi^{(2)}_{\alpha(i)} \colon M \to M $$
    is a sequence of contactomorphisms that squeezes $K_2$ onto $K_0$.
    \item The property of either an isotopy or a sequence of contactomorphisms to squeeze a submanifold $K_0$ onto $K$ does not depend on the choice of metric.
\end{enumerate}
\end{lem}

We establish the non-squeezing result for Legendrian knots onto transverse knots.

  \begin{mainthm}
    \label{thm:TransverseNonSqueezing}
  Let $\Lambda \subset (M,\xi)$ be a Legendrian knot. If $T \subset (M,\xi)$ is a transverse knot, then there does not exist any sequence of contactomorphisms that squeezes $\Lambda$ onto $T$ (see Definition \ref{defn:squeezing}).
\end{mainthm}

In the case when $H_1(M)=H_2(M)=0$, so that the Thurston--Bennequin number of any Legendrian knot is well-defined, this non-squeezing result can be seen to follow directly from Theorem \ref{thm:girouxhonda}. 
For the general statement, the main ingredient is the Thurston--Bennequin inequality for Legendrian knots in standard $\R^3$ proven by Bennequin \cite{Bennequin:Entrelacements} (or, more precisely, a relative formulation for unknotted Legendrian cores of the solid torus $J^1S^1$).

In combination with the existence of squeezing of non-Legendrians onto transverse knots proven by Theorem \ref{thm:squeezing} above, we obtain the following non-squeezing for Legendrians into a neighborhood of a non-Legendrian.

\begin{maincor}
  \label{cor:nonsqueezing}
  Let $\Lambda \subset (M^3,\xi)$ be a Legendrian knot. If $K \subset (M^3,\xi)$ is a non-Legendrian knot, then there does not exist a sequence of contact embeddings $\varphi_i \colon M \to M$ that squeezes $\Lambda$ onto $K$.
\end{maincor}
\begin{proof}
  Assume that there exists a sequence of contactomorphisms that squeezes $\Lambda$ onto $K$. Apply Theorem \ref{thm:squeezing} to produce a contact isotopy that squeezes $K$ onto a transverse knot $T$. By Lemma \ref{lem:transitivity} we can find a sequence of contactomorphisms that squeezes $\Lambda$ onto $T$; this is in contradiction with Theorem \ref{thm:TransverseNonSqueezing}.
\end{proof}

\begin{rem}
  In contact manifolds of dimension $2n+1\ge 5$ the result analogous to Corollary \ref{cor:nonsqueezing} does not hold: there are contact isotopies that squeeze certain Legendrians onto non-Legendrians. Such examples can be constructed by alluding to Murphy's $h$-principle for loose Legendrians \cite{LooseLeg}. Namely, by this $h$-principle we can approximate any $n$-dimensional non-Legendrian submanifold by a loose Legendrian while keeping control of its formal Legendrian isotopy class. The loose Legendrian approximations are moreover Legendrian isotopic by the same $h$-principle.

The main difference between high dimensions and dimension $2n+1=3$ in this respect is that, in the low dimensional case, one cannot add stabilizations inside a sufficiently small neighborhood of a transverse knot (or, more, generally non-Legendrian knot) without decreasing the relative Thurston--Bennequin number.
\end{rem}

In symplectic geometry the existence of capacities for Lagrangian submanifolds defined by Floer homology has given rise to many rigidity phenomena of a quantitative nature. In particular, in \cite{LaudenbachSikorav:HamiltonianDisjunction} Laudenbach--Sikorav showed that Lagrangians cannot be placed inside neighborhoods of non-Lagrangians. This result can be used to show that smooth limits of Lagrangians under a sequence of symplectomorphism that converge to a homeomorphism must again be Lagrangian. The analogous result for coisotropic manifolds was shown in codimension one by Opshtein \cite{Opshtein:C0}. The full answer was later given by Humili\`ere--Leclercq--Seyfaddini who established the analogous result for arbitrary coisotropic submanifolds in \cite{Humiliere:Coisotropic}.
The analogous questions in the setting of contact topology have only seen partial results \cite{Nakamura:C0, RosenZhang,  Usher:C0}.
Using the above non-squeezing result we settle the question in dimension three.
\begin{mainthm}
\label{thm:c0}
Let $(M^3,\xi)$ be a three-dimensional contact manifold and $\varphi_i \in Cont(M,\xi)$ a sequence of contactomorphisms that converge in $C^0$-norm $\varphi_i \to_{C^0} \varphi_\infty$ to a homeomorphism $\varphi_\infty$. Let $\Lambda \subset (M,\xi)$ be a Legendrian knot whose image $\varphi_\infty(\Lambda)$ is a smooth knot. Then $\varphi_\infty(\Lambda)$ is Legendrian as well. In addition, there exists a globally defined smooth contactomorphism of $M$ that maps $\Lambda$ to $\varphi_\infty(\Lambda)$.
\end{mainthm}
Nakamura proves the first statement in Theorem \ref{thm:c0} for arbitrary dimension assuming that for some contact form there exists a uniform lower bound on the lengths of the Reeb chords from $\varphi_i(\Lambda)$ to itself \cite[Theorem 3.4]{Nakamura:C0}. He also assumes some technical conditions that we have since lifted \cite[Corollary 1.5]{DRS3}. Rosen and Zhang prove the first part of Theorem \ref{thm:c0} in arbitrary dimensions assuming a uniform convergence of the conformal factors $f_i$  (defined by $\varphi_i^*\alpha =f_i \alpha$ for contact form $\alpha$) \cite[Theorem 1.4]{RosenZhang}. Usher generalizes Rosen and Zhang's result assuming certain lower bounds on the $f_i$ \cite[Theorem 1.2]{Usher:C0}. Observe that the latter works do not make any claims about the contactomorphism type of the limit.

Since any tangent vector in the contact plane can be realized as the tangent to a small Legendrian knot, our result Theorem \ref{thm:c0} is strong enough to settle ``Gromov's Alternative'' in this dimension: a smooth $C^0$-limit of contactomorphisms is itself a contactomorphism. This result was first proven by Eliashberg \cite{Eliashberg:GromovsAlternative}; see work by M\"{u}ller--Spaeth for a more recent proof \cite{MullerSpaeth:GromovsAlternative}. Note that, in the case when the $C^0$-limit homeomorphism $\varphi_\infty$ moreover is smooth, Gromov's alternative can itself be used to derive the conclusion Theorem \ref{thm:c0}.

\section{Transitivity of squeezing (Proof of Lemma \ref{lem:transitivity})}
\label{sec:TransitivityProof}

We prove Lemma \ref{lem:transitivity}.

Part (i): Consider the contact Hamiltonian $H_t \colon M \to \R$ that generates the contact isotopy $\psi_t$. We cut off $H_t$ via a sequence of bump functions $\rho_t\cdot H_t$ that have support contained inside $B_{\epsilon(t)}(K)$ 
for all $t \ge 0$, while $\rho_t \equiv 1$ holds near $\psi_t(K_0)$. The new contact isotopy $\varphi_t$ obtained  restricts to the old isotopy along $K_0$, and hence squeezes $K_0$ onto $K$ as well.

  The corresponding sequence of contactomorphisms $\varphi_i$ for the integer times $i=0,1,2,3,\ldots$ is the sought sequence that squeezes $K_1$ onto $K$. For Part (2) of Definition \ref{defn:squeezing}, we may take 
$$i_{r,\epsilon} \coloneqq \min \{ i_0; \:\: \epsilon(i) < r \: \text{for all}\: i \ge i_0 \}$$
to be independent of $\epsilon.$ 
In this case, the maps $\varphi_i \circ \varphi_{j}^{-1}$ with $i \ge j \ge i_{r,\epsilon}$ all have support contained inside $B_r(K)$, i.e.~$\varphi_i \circ \varphi_{j}^{-1}(x)=x$ for $x \notin B_r(K)$.

Part (ii): By the assumption that $\varphi^{(\nu)}_{i}$ are sequences that squeeze $K_\nu$ onto $K_{\nu-1}$ we get that, for any $r, \epsilon>0$, there are $i^{(\nu)}_{r,\epsilon}$ such that
$$d(\varphi^{(\nu)}_i\circ (\varphi^{\nu}_j)^{-1}(x),x) < \epsilon$$
holds for all $x \notin B_r(K_{\nu-1})$ 
and $i \ge j \ge i^{(\nu)}_{r,\epsilon}$. 
In particular,
$$(\varphi_{j}^{(1)})^{-1}(B_{r-\delta}(K_0)) \subset (\varphi_{i}^{(1)})^{-1}(B_{r}(K_0)) \subset (\varphi_{j}^{(1)})^{-1}(B_{r+\delta}(K_0))$$
 may be assumed to hold for all sufficiently small $\delta>0$ and $i \ge j \ge i^{(1)}_{r/2,\epsilon/2}$. By the definition of squeezing, we can assume that $K_1 \subset (\varphi_{j}^{(1)})^{-1}(B_{r-\delta}(K_0))$ is satisfied after increasing $i_{r/2,\epsilon/2} \gg 0$ further and taking $j \ge i_{r/2,\epsilon/2}$. In other words, all images $\varphi_{i}^{-1}(B_{r}(K_0))$ can be assumed to contain a fixed neighborhood $(\varphi_{i_{r/2,\epsilon/2}}^{(1)})^{-1}(B_{r-\delta}(K_0)) \supset K_1$ whenever $i \ge i_{r/2,\epsilon/2}$.

 We claim that the sequence $\varphi^{(1)}_i \circ \varphi^{(2)}_{\alpha(i)}$ squeezes $K_2$ onto $K_0$ for a suitable increasing re-indexing $\alpha(i) \ge i$ where $\alpha(i)-i \gg 0$ is taken to be sufficiently large.

First we verify that Part (1) of the definition is satisfied. Note that we have an inclusion,
$$\varphi^{(2)}_{\alpha(i)}(K_2) \subset B_{\epsilon^{(2)}_{\alpha(i)}}(K_1)$$ 
where the sequence $\epsilon^{(2)}_{\alpha(i)}$ satisfies $\lim_{i \to +\infty} \epsilon^{(2)}_{\alpha(i)}=0$. Consequently, $B_{\epsilon^{(2)}_{\alpha(i)}}(K_1) \subset (\varphi^{(1)}_{j})^{-1}(B_{r}(K_0))$ may be assumed to hold for any arbitrary $r >0$ and all $i \gg 0$, whenever $j \gg i^{(1)}_{r,\epsilon}$. In conclusion,
 for any $r>0,$
$$ \varphi^{(1)}_i \circ \varphi^{(2)}_{\alpha(i)}(K_2) \subset B_{r}(K_0)$$
is satisfied whenever we take $\alpha$ to satisfy $\alpha(i)-i \gg 0$. The image of $K_2$ is moreover smoothly isotopic to $K_0$ inside the same subset.
 
What remains is to verify Part (2) of the definition. Take
$$i_{r,\epsilon} \coloneqq \max(i^{(1)}_{r,\epsilon/4},i^{(2)}_{\rho(\epsilon),\epsilon/4}),$$
for $\rho(\epsilon)>0$ sufficiently small so that the inclusion
  $$B_{\rho(\epsilon)}(K_0) \subset (\varphi_{i}^{(1)})^{-1}(B_r(K_0))$$
  is satisfied for all $i \ge i^{(1)}_{r,\epsilon/4}$. It is then readily checked that Part (2) is satisfied for the sequence $\{\varphi^{(1)}_i \circ \varphi^{(2)}_{\alpha(i)}\}$ of contactomorphisms.

Part (iii): This is obvious since the property of convergence is independent of the metric, as it only depends on the topology.  
\qed
  
\section{Some ${\tt tb}$ prerequisites (Proof of Theorem \ref{thm:TransverseNonSqueezing})}

The material in this section concerns a type of non-squeezing behavior for Legendrians that can roughly be described as follows: a Legendrian that approximates a transverse knot sufficiently well (in a certain technical sense) can be destabilized. This matches well with the intuition that one needs to add zig-zags in order to approximate non-Legendrian knots by Legendrians; see Figure \ref{fig:stabilized}. As said in the introduction, this result is implicitly contained in the proofs of the classification of contact structures on solid tori from \cite{MR1779622}, \cite{Honda:Classification}. However, we choose a different path here, and instead prove the result by directly relying only on the Thurston--Bennequin inequality for Legendrian knots in tight three-manifolds. 
Recall that the Thurston--Bennequin inequality \cite{Bennequin:Entrelacements} for Legendrian unknots $\Lambda \subset (S^3,\xi_{st})$ in the standard contact sphere states that
$$ {\tt tb}(\Lambda) \le -1.$$
This is a strong result that e.g.~implies the tightness of the standard sphere. We start by recalling certain topological notions in contact manifolds, such as the Thurston--Bennequin number.

\subsection{Twisting and Thurston--Bennequin}
\label{sec:twisting}

Define the {\bf linking number} of two disjoint oriented null-homologous knots $K_0 \sqcup K_1 \subset M^3$ by the algebraic intersection number
$$ {\tt lk}(K_0,K_1) \coloneqq K_0 \bullet \Sigma $$
where $\Sigma$ is a choice of two-chain with boundary $\partial \Sigma=K_1.$ When the ambient manifold satisfies $H_2(M)=0$ this linking number does not depend on the choice of null-homology.

A {\bf framing} of a knot $K \subset M^3$ inside an orientable three-dimensional manifold can be defined either as a non-vanishing normal vector field, or as a small piece of an embedded orientable surface $\Sigma$ whose boundary contains the knot. Recall that two different framings of an oriented knot have a well-defined winding number in $\Z$, which vanishes if and only if the two framings are homotopic. This winding number can be interpreted as the ``difference of framings''
 via the formula
$$ d(\OP{Fr}_{\Sigma_0},\OP{Fr}_{\Sigma_1}) \coloneqq K_{\Sigma_0} \bullet \Sigma_1 \in \Z,$$
where $K_{\Sigma_0}$ is a sufficiently small push-off of $K$ along a non-vanishing normal vector field that is tangent to the surface $\Sigma_0$. Here $\Sigma_i$ are given orientations that agree on the boundary component $K$; it thus follows that the above number only depends on the orientation of the ambient three-manifold. In the case when $\Sigma_1$ is embedded and $K=\partial \Sigma_1$ is its entire boundary, we get the identity $d(\OP{Fr}_{\Sigma_0},\OP{Fr}_{\Sigma_1})={\tt lk}(K_{\Sigma_0},K).$ 

Recall that a contact structure on a three-dimensional manifold induces a canonical orientation via the locally defined volume form $\alpha \wedge d\alpha$. A Legendrian knot $\Lambda$ has the canonical framing $\OP{Fr}_{\OP{Reeb}}$ given by push-off in the Reeb direction. In the case when $\Lambda \subset \R^3$ we moreover have the canonical Seifert framing induced by a bounding surface $\Sigma_\Lambda$. We define the {\bf Thurston--Bennequin} number via
$$ {\tt tb}(\Lambda) \coloneqq d(\OP{Fr}_{\OP{Reeb}},\OP{Fr}_{\OP{Seifert}})={\tt lk}(\Lambda_{\OP{Reeb}},\Lambda),$$
where $\Lambda_{\OP{Reeb}}$ denotes a small push-off in the Reeb direction. In arbitrary contact manifolds one can define the Thurston--Bennequin number by a similar formula when the knot is null-homologous; in general, this number depends on a choice of null-homology. In addition, given a fixed knot $K \subset M,$ we can define a {\bf relative Thurston--Bennequin number} for any Legendrian knot $\Lambda \subset M \setminus K$ that satisfies $\{\pm[\Lambda]\}=\{\pm[K]\} \subset H_1(M)$. Again, this number depends on the choice of a chain $\Sigma$ with $\partial \Sigma = \Lambda \cup K$ in general; we denote it by
$$ {\tt tb}_{K,\Sigma}(\Lambda) \coloneqq \Lambda_{\OP{Reeb}} \bullet \Sigma.$$
This number is invariant under contactomorphisms $\phi$ in the sense that
$$ {\tt tb}_{K,\Sigma}(\Lambda)={\tt tb}_{\phi(K),\phi(\Sigma)}(\phi(\Lambda)).$$
When $H_2(M)=0$ it immediately follows that ${\tt tb}_{K,\Sigma}(\Lambda)$ is independent of the choice of chain $\Sigma,$ in which case we will simply write ${\tt tb}_K(\Lambda)$.

When $\Lambda$ is either contained in a surface $\Sigma$, or equal to one of its boundary components, one can define the following quantity related to the Thurston--Bennequin number. The {\bf twisting number} is given by
$${\tt tw}(\Lambda,\OP{Fr}_\Sigma) \coloneqq d(\OP{Fr}_{\OP{Reeb}},\OP{Fr}_{\Sigma}) \in \Z.$$
where $Fr_\Sigma$ is the framing induced by the surface. When $\partial \Sigma=\Lambda$ we immediately get
$${\tt tw}(\Lambda,\OP{Fr}_\Sigma)={\tt tb}_\Sigma(\Lambda) :={\tt tb}_{\emptyset,\Sigma}(\Lambda)$$
where the right-hand side is the non-relative Thurston--Bennequin number.

The following results are standard.
\begin{lem}
  \label{lem:twisting1}
  \begin{enumerate}
    \item Suppose $K \subset (M^3,\xi)$ is a smooth knot, $\Lambda \subset M \setminus K$ is a Legendrian knot,  $\Sigma'$ is a (possibly) singular chain with $\partial \Sigma'=\Lambda \sqcup K$ an oriented link,  and $\Sigma''$ is a singular chain with $K=-\partial \Sigma''.$ Then
$${\tt tb}_{\Sigma' \cup \Sigma''}(\Lambda)={\tt tb}_{\Sigma',K}(\Lambda)+ \Lambda \bullet \Sigma''={\tt tb}_{\Sigma',K}(\Lambda)-{\tt lk}(\Lambda,K).$$

\item Let $\Sigma \subset (M^3,\xi)$ be an oriented embedded surface with boundary
$$ \partial \Sigma=\Lambda_0 \sqcup -\Lambda_1$$
an oriented Legendrian link. Then
$$ {\tt tb}_{\Sigma',K}(\Lambda_1)-{\tt tb}_{\Sigma \cup \Sigma',K}(\Lambda_0)={\tt tw}(\Lambda_1,\OP{Fr}_{\Sigma})-{\tt tw}(\Lambda_0,\OP{Fr}_\Sigma)$$
where $K \subset M \setminus \Sigma$ is either a knot or the empty set $K=\emptyset$, and $\Sigma'$ is a singular chain that satisfies $\partial \Sigma'=\Lambda_1 \cup K$.
\end{enumerate}
\end{lem}
\begin{proof}
  Part (1): This is a straightforward computation of algebraic intersection numbers.

  Part (2): First we use the fact that $\Sigma$ is embedded in order to compute
  \begin{equation}
    \label{eq1}
    {\tt tb}_{\Sigma \cup \Sigma',K}(\Lambda_0)={\tt tw}(\Lambda_0,\OP{Fr}_\Sigma)+(\Lambda_0)_{\OP{Reeb}} \bullet \Sigma'
  \end{equation}
  where the second term counts intersections of $(\Lambda_0)_{\OP{Reeb}}$ and $\Sigma'$.

Note that the push-off $\Sigma_{\OP{Reeb}}$ in the Reeb-direction is an embedded homology between $(\Lambda_0)_{\OP{Reeb}}$ and $(\Lambda_1)_{\OP{Reeb}}$. We will analyze the intersection locus
  $\Sigma_{\OP{Reeb}} \cap \Sigma'.$
  For simplicity we consider the case when the chain $\Sigma'$ is an immersed surface. For $\Sigma_{\OP{Reeb}}$ a sufficiently small push-off, followed by a small generic perturbation, the intersections consist of a union of oriented paths in $\Sigma_{\OP{Reeb}}$ whose boundary points transversely intersect the boundary
  $$\partial \Sigma_{\OP{Reeb}}=(\Lambda_0)_{\OP{Reeb}}-(\Lambda_1)_{\OP{Reeb}},$$
  except for a number of boundary components that are in bijection with the finite number of transverse intersection points
  $$\partial \Sigma' \cap \Sigma_{\OP{Reeb}} = \Lambda_1 \cap \Sigma_{\OP{Reeb}} \subset \Sigma_{\OP{Reeb}} \setminus \partial \Sigma_{\OP{Reeb}}$$
in the interior of $\Sigma_{\OP{Reeb}}$.
A signed count of these different boundary points gives rise to the identity
  $$ (\Lambda_0)_{\OP{Reeb}} \bullet \Sigma' = (\Lambda_1)_{\OP{Reeb}} \bullet \Sigma' -  \Lambda_1 \bullet \Sigma_{\OP{Reeb}}$$
  of algebraic intersection numbers.

  In the latter equation, the first term on the right-hand side is equal to ${\tt tb}_{\Sigma',K}(\Lambda_1)$, while the second term is equal to
  $$-(\Lambda_1)_{\OP{Reeb}} \bullet \Sigma=-{\tt tw}(\Lambda_1,\OP{Fr}_\Sigma),$$ 
  where we again have used the fact that $\Sigma$ is embedded. To conclude:
  $$(\Lambda_0)_{\OP{Reeb}} \bullet \Sigma'={\tt tb}_{\Sigma',K}(\Lambda_1)-{\tt tw}(\Lambda_1,\OP{Fr}_\Sigma)$$
  which gives the sought equality between Thurston--Bennequin and twisting numbers when combined with Equation \eqref{eq1}.
\end{proof}

From Part (2) of the previous lemma we immediately deduce the following.
\begin{cor}
  \label{cor:tbdifference}
  Let $\Lambda_0,\Lambda_1 \subset (M,\xi)$ be two Legendrian knots inside a contact manifold $M$ that satisfies $H_2(M)=0$, where $\Lambda_1 \sqcup \Lambda_2=\partial \Sigma$ is the boundary of an embedded orientable surface $\Sigma \subset M$. For any knot 
  $$K \subset M \setminus \Sigma$$
  in the same homology class (we allow $K=\emptyset$), the difference
  $$ {\tt tb}_K(\Lambda_0) - {\tt tb}_K(\Lambda_1)$$
  of relative Thurston--Bennequin numbers is independent of the choice of such $K$
\end{cor}

The crucial technical result that we rely on is the following relative version of the Thurston--Bennequin inequality:

\begin{lem}[Bennequin \cite{Bennequin:Entrelacements}]
  \label{lem:tbineq}
  Consider a Legendrian knot $\Lambda \subset J^1S^1$ which is smoothly isotopic to the zero section $j^10$, and fix a reference Legendrian $K=j^1c$ for $c \gg 0$. It follows that
  the relative Thurston--Bennequin invariants satisfy
  $$ {\tt tb}_K(\Lambda) \le {\tt tb}_K(j^10)=0,$$
  i.e.~the zero-section has maximal relative Thurston--Bennequin invariant.
\end{lem}
  \begin{proof}
Construct a contact embedding
  $$F \colon (J^1S^1,\xi_{st}) \hookrightarrow (\R^3,\xi_{st})$$
  that takes the one-jet $j^1C$ of a constant function to a standard Legendrian unknot, i.e.~a knot which is Legendrian isotopic to $\Lambda_{st} \subset (\R^3,\xi_{st})$ with ${\tt tb}(\Lambda_{st})=-1$. Using this we immediately compute
  $${\tt lk}(F(j^1C),F(j^10))={\tt tb}(\Lambda_{st})=-1$$
  for any $C >0.$ For $c \gg 0$ we thus get ${\tt lk}(F(j^1c),F(\Lambda))=-1$ as well, since $\Lambda$ and $j^10$ can be assumed to be smoothly isotopic inside $J^1S^1 \setminus j^1c$.

  The image $F(\Lambda)$ is also a Legendrian unknot. Consider an embedded annulus $\Sigma' \subset F(J^1S^1)$ with boundary $\partial \Sigma'=F(K) \cup F(\Lambda)$, and let $\Sigma'' \subset \R^3$ be a null-homology of $F(K)=F(j^1c).$

Alluding to Part (1) of Lemma \ref{lem:twisting1} with $F(K)=F(j^1c)$, $c \gg 0$, and $\Sigma'$ and $\Sigma''$ as above, we conclude
$$ {\tt tb}_K(\Lambda)={\tt tb}(F(\Lambda))-{\tt lk}(F(j^1c),F(\Lambda))={\tt tb}(F(\Lambda))+1.$$
In particular, we get
$${\tt tb}_K(j^10)={\tt tb}(\Lambda_{st})+1=0.$$

Finally, the Thurston--Bennequin inequality \cite{Bennequin:Entrelacements} gives
$$ {\tt tb}(F(\Lambda)) \le {\tt tb}(\Lambda_{st})=-1$$
from which the sought inequality follows.
\end{proof}

\subsection{Standard Legendrians near a transverse knot}
\label{sec:Lambdak}

In this subsection we analyze the standard contact solid tori
  $$\partial D^2_{\sqrt{2/k}} \times S^1 \subset \left(\RR^2_{(x,y)} \times S^1_\theta, \ker\left(d\theta-(1/2)\left(y\,dx-x\,dy\right)\right)\right)$$
which for integers $k$ are foliated by the Legendrian knots
$$\Lambda_k \coloneqq \left\{\left(\sqrt{2/k}e^{ik\theta},\theta+\theta_0\right); \:\theta \in S^1\right\} \subset \R^2 \times S^1$$
that are smoothly isotopic to the core $T=\{0\} \times S^1$ of the solid torus, which is a transverse knot.

\begin{lem}
  \label{lem:prequant}
  There is a contact-form-preserving contact embedding of
  $$(B^2_{\sqrt{2}} \times S^1_\theta,d\theta-(1/2)(xdy-ydx))$$
  into $(S^3,x\,dy-y\,dx)$, with image being the complement of a standard transverse unknot. This embedding, moreover, takes $\Lambda_{2}$ to the standard Legendrian unknot with ${\tt tb}=-1$.

  It follows that, for any fixed $K \subset \R^2 \times S^1$ that satisfies $[K]=[\{0\} \times S^1] \in H_1(\R^2 \times S^1)$, the relative Thurston--Bennequin invariant
  $${\tt tb}_K(\Lambda), \:\:\text{for}\:\: \Lambda \subset (\R^2 \times S^1)\setminus K, \:\: [\Lambda]=[K] \in H_1(\R^2 \times S^1),$$
  satisfies the bound
  $${\tt tb}_K(\Lambda) \le -1+{\tt lk}(\Lambda,K)$$
  whenever $\Lambda$ is smoothly isotopic to $\{0\} \times S^1$.
  \end{lem}
\begin{proof}
  Recall that $(S^3,x\,dy-y\,dx)$ is foliated by periodic Reeb orbits of length $2\pi$, which gives it the structure of the prequantization $S^1$-bundle over $\CP^1$ with curvature $2\pi.$ The complement of a single fibre of this prequantum bundle can thus be identified with the trivial prequantum bundle
  $$\left(B_{\sqrt{2}}^2 \times S^1,\ker\left(d\theta-(1/2)\left(y\,dx-x\,dy\right)\right)\right) \to B_{\sqrt{2}}.$$
The standard Legendrian unknot in $S^3$ can be realized as the intersection $S^3 \cap \Re \C^2$, and can thus be seen to be the two-fold cover of the equator in the prequantum bundle projection $S^3 \to \CP^1$. Since $\Lambda_2$ lives over a disc of total area $\pi$, it can be identified with the unknot in the above chart $B_{\sqrt{2}}^2 \times S^1$.

  Similarly to the proof of Lemma \ref{lem:tbineq}, the uniform upper bound then follows from the Thurston--Bennequin inequality in for Legendrian unknots in $S^3$ together with Part (1) of Lemma \ref{lem:twisting1}. Note that the linking number ${\tt lk}(\Lambda,K)$ computed in $S^3$ does not depend on the choice of $\Lambda$ as above.
\end{proof}

  
\begin{lem}
  \label{lem:tbnumbers}
  Take any reference knot $K \subset (\R^2 \setminus D^2_{\sqrt{2/m}}) \times S^1$ which is homologous to $\{0\} \times S^1$. For $m \le k,$ the Legendrian knots
  $$\Lambda_k \subset \partial D^2_{\sqrt{2/k}} \times S^1 \:\:\text{and}\:\: \Lambda_m \subset \partial D^2_{\sqrt{2/m}} \times S^1 $$
  satisfy
  $$ {\tt tb}_K(\Lambda_m) -{\tt tb}_K(\Lambda_k) = -(m-k)$$
  Hence, it follows that
  $$ {\tt tb}_K(\Lambda_m) -{\tt tb}_K(\Lambda) \ge -(m-k)$$
  for any Legendrian $\Lambda$ which is contained inside a standard neighborhood of $\Lambda_k$ while, moreover, being smoothly isotopic to $\Lambda_k$ inside the same neighborhood.
\end{lem}
\begin{proof}

  We begin by establishing the relation
  $${\tt tb}_K(\Lambda_m) -{\tt tb}_K(\Lambda_k) = -(m-k)$$
  between relative Thurston--Bennequin numbers. For this we use the contact embedding
  $$B^2_{\sqrt{2}} \times S^1 \hookrightarrow (S^3,\ker(xdy-ydx))$$
  provided by Lemma \ref{lem:prequant}. Since $\{0\} \times S^1$ bounds a disc in the prequantization bundle that has intersection number $-k$ with $\Lambda_k$, one readily computes ${\tt tb}(\Lambda_k)=-k+1$ inside the prequantization space $S^3 \to \CP^1$; it is the intersection number of a curve of slope $k$ and $1$  
 on the torus. (Note that, in particular, $\Lambda_2$ is the standard unknot.) The sought relation for the relative Thurston--Bennequin numbers then follows from Corollary \ref{cor:tbdifference}.

We continue with the inequality
  $$ {\tt tb}_K(\Lambda_m) -{\tt tb}_K(\Lambda) \ge -(m-k).$$
Note that there exists a smoothly embedded cylinder
 $\Sigma \subset D^2_{\sqrt{2/m}} \times S^1$  with boundary $\partial \Sigma=\Lambda_m \sqcup \Lambda_k$; hence such a cylinder with boundary equal to $\Lambda_m \sqcup \Lambda$ also exists. Corollary \ref{cor:tbdifference} now implies that each of the differences
 $${\tt tb}_K(\Lambda_m) -{\tt tb}_K(\Lambda) \:\: \text{and}\:\: {\tt tb}_K(\Lambda_m) -{\tt tb}_K(\Lambda_k)$$
 are independent on the choice of reference knot
 $$K \subset \left(\R^2 \setminus D^2_{\sqrt{2/m}}\right) \times S^1$$
 In particular, Lemma \ref{lem:tbineq} shows that
 $${\tt tb}_K(\Lambda_k)-{\tt tb}_K(\Lambda) \ge 0.$$
One can now compute
\begin{eqnarray*}
  \lefteqn{{\tt tb}_K(\Lambda_m) -{\tt tb}_K(\Lambda) =}\\
  &=& {\tt tb}_K(\Lambda_m) -{\tt tb}_K(\Lambda_k)+({\tt tb}_K(\Lambda_k)-{\tt tb}_K(\Lambda)) \\
  &\ge& {\tt tb}_K(\Lambda_m) -{\tt tb}_K(\Lambda_k) \ge -(m-k)
\end{eqnarray*}
as sought.
  \end{proof}

\subsection{Non-squeezing results for Legendrian knots into neighborhoods of transverse knots}
\label{sec:TransverseNonSqueezing}

In this subsection we can finally prove Theorem \ref{thm:TransverseNonSqueezing}.

We argue by contradiction and assume that there exists $\varphi_i \colon M \to M,$ such that in the language of Definition \ref{defn:squeezing}, the sequence
  \begin{gather*}
    \varphi^{(1)}_i \coloneqq \varphi_i \circ \varphi_{j_0}^{-1} \colon M \to M, \:\: i \ge j_0,
  \end{gather*}
 of contactomorphisms for $j_0 \coloneqq i_{r/2,\epsilon/3}$ squeezes the Legendrian $\Lambda' \coloneqq \varphi_{j_0}(\Lambda)$ onto $T$. By the definition of $i_{r/2,\epsilon/3}$ it follows that $d(\varphi_i^{(1)}(x),x)<\epsilon/3$ holds on the subset $M \setminus B_{r/3}(T)$ whenever $i\ge j_0$. After increasing $j_0 \gg 0, j_0 \ge i_{r/2,\epsilon/3}$ even further, we may also assume that $\Lambda' \subset B_{r/3}(T)$ is satisfied for the same choice of $r >0$.

  By Part (iii) of Lemma \ref{lem:transitivity} the property of being a squeezing sequence does not depend on the choice of metric. After choosing an appropriate metric on $M,$ 
 and taking $r>0$ above to be sufficiently small, the transverse neighborhood theorem implies that one can find a
  neighborhood $U \subset M$ of the transverse knot $T \subset M$ that is contactomorphic to
  \begin{gather*}
    \left(B^2_{2r} \times S^1_\theta, \ker\left(d\theta-(1/2)\left(y\,dx-x\,dy\right)\right)\right),\end{gather*}
  under which $T$ is identified with $\{0\} \times S^1$ and $B_s(T)$  is identified with $B^2_s \times S^1$ for all $s \le 2r$. Note that, by the above, we may assume that
  $$B_{r/2}(T) \subset \varphi_i^{(1)}(B_{r}(T)) \subset B_{2r}(T) \hookrightarrow \R^2 \times S^1.$$

There is a compactly supported contact isotopy of $B^2_r \times S^1$ that squeezes the transverse knot $\{0\} \times S^1$ onto any of the Legendrian knots $\Lambda_k$ inside the same neighborhood, where the knots $\Lambda_k \subset \partial D^2_{\sqrt{2/k}} \times S^1$
 were described in Section \ref{sec:Lambdak} above. Namely, one can use the explicitly constructed isotopy
  $$\Lambda_{k,t} \coloneqq \left\{\left(t\sqrt{2/k}e^{ik\theta},\theta+\theta_0\right); \:\theta \in S^1\right\} \subset \R^2 \times S^1$$
  which is through transverse knots for all $t \in [0,1)$ (at $t=1$ the embedding becomes equal to the Legendrian knot $\Lambda_k$). Here we need to use the standard fact that transverse isotopies are generated by an ambient contact isotopy $\psi_t$;  see Corollary \ref{cor:isoext}. Note that $\psi_t$ can be assumed to be supported inside $B_{\sqrt{2/k}}(T).$ Below we will take $k \gg 0$.

Consider the sequence $\varphi^{(2)}_i$ of contactomorphisms that is produced by Part (i) of Lemma \ref{lem:transitivity} applied to the above contact isotopy $\psi_t$ that squeezes $T$ onto $\Lambda_k$. Part (ii) of Lemma \ref{lem:transitivity} applied to the sequences $\varphi_i^{(1)}$ and $\varphi^{(2)}_i,$ 
 i.e.~the transitivity of the existence of squeezing sequences, implies that there is a sequence of contactomorphisms $\varphi_i^{(1)} \circ \varphi^{(2)}_{\alpha(i)} \colon M \to M$ that squeezes $\Lambda' \subset B_{r/3}(T)$ onto $\Lambda_k$. Note that, by Part (i) of Lemma \ref{lem:transitivity}, after choosing $k \gg 0$ in order for $\sqrt{2/k} < r/4$ to hold, we can assume that $\varphi^{(2)}_{\alpha(i)}|_{M \setminus B_{r/4}(T)}=\id_M.$
  
 The proof consists of computations and estimates of relative Thurston--Bennequin numbers ${\tt tb}_{K''}(\Lambda'')$ for Legendrians $\Lambda'' \subset B_{r/2}(T)$ and smooth knots $K'' \subset B_{2r}(T) \setminus B_{r/2}(T)$, where $[K'']=[\Lambda''] \in H_1(B_{2r}(T))=\Z$ are generators of the first homology. The relative Thurston--Bennequin number in general depends on a choice of two-chain. However, we will always consider these relative Thurston--Bennequin numbers as defined inside the contact manifold $B_{2r}(T) \cong B^2_{2r} \times S^1$; since $H_2(B_{2r}(T))=0$ these numbers are well-defined (depending only on $K''$).

Fix an arbitrary smooth knot $K \subset \partial \overline{B_{r/2}(T)}$ for which $[K]=[T] \in H_1(\R^2 \times S^1)$. We start by finding an estimate for the relative Thurston--Bennequin number ${\tt tb}_{K}(\Lambda')$ (where this invariant is computed inside the contact manifold $B_{2r}(T)$). Since $\varphi_i^{(1)} \circ \varphi^{(2)}_{\alpha(i)}$ squeezes $\Lambda'$ onto $\Lambda_k$, Lemma \ref{lem:tbnumbers} implies that
$$ {\tt tb}_{K}(\Lambda_m)-{\tt tb}_{K}(\varphi_i^{(1)} \circ \varphi^{(2)}_{\alpha(i)}(\Lambda')) \ge -(m-k)$$
whenever $\Lambda_k,\Lambda_m \subset B_{r/2}(T)$ and $i \gg 0$ is sufficiently large; to that end we note that, for large $i$, $\varphi_i^{(1)} \circ \varphi^{(2)}_{\alpha(i)}(\Lambda')$ is contained inside a standard contact neighborhood $J^1\Lambda_k \hookrightarrow \R^2 \times S^1$ of $\Lambda_k$, in which $\varphi_i^{(1)} \circ \varphi^{(2)}_{\alpha(i)}(\Lambda')$ moreover is isotopic to $j^10=\Lambda_k$.

It now follows that
$$ {\tt tb}_{(\varphi_i^{(1)} \circ \varphi^{(2)}_{\alpha(i)})^{-1}(K)}((\varphi_i^{(1)} \circ \varphi^{(2)}_{\alpha(i)})^{-1}(\Lambda_m))-{\tt tb}_{(\varphi_i^{(1)} \circ \varphi^{(2)}_{\alpha(i)})^{-1}(K)}(\Lambda') \ge -(m-k)$$
for $i \gg 0$ large. Note that
$$(\varphi_i^{(1)} \circ \varphi^{(2)}_{\alpha(i)})(B_{r}(T)) \supset B_{r-\delta}(T),$$
which means that the latter inequality is between relative Thurston--Bennequin numbers computed in $B_{2r}(T)$, and where $(\varphi_i^{(1)} \circ \varphi^{(2)}_{\alpha(i)})^{-1}(\Lambda_m) \subset B_{r-\delta}(T)$.

Corollary \ref{cor:tbdifference} implies that
$$ {\tt tb}_{K'}(\Lambda') -(m-k) \le {\tt tb}_{K'}((\varphi_i^{(1)} \circ \varphi^{(2)}_{\alpha(i)})^{-1}(\Lambda_m))$$
holds for all $K' \subset B_{2r}(T) \setminus B_{r-\delta}(T)$ in the homology class $[K']=[T]$. Taking $k \to +\infty$ while keeping $m >0$ and $K'$ fixed implies that the right-hand side tends to $+\infty$. In other words, the Legendrian $(\varphi_i^{(1)} \circ \varphi^{(2)}_{\alpha(i)})^{-1}(\Lambda_m)$ that is isotopic to $T$ can be assumed to have a relative Thurston--Bennequin number that is greater than the upper bound from Lemma \ref{lem:prequant}, 
which is a contradiction
on the Thurston--Bennequin numbers in $B^2_{2r} \times S^1$ for Legendrians in the same smooth isotopy class as $T$.

\section{Normal neighborhood for non-Legendrians (Proof of Theorem \ref{thm:squeezing})}

Here we establish a normal neighborhood theorem for non-Legendrian knots. The goal is to use the standard neighborhood for proving the existence of squeezing for non-Legendrians onto some transverse knot as stated in Theorem \ref{thm:squeezing}. Throughout this section, $(M^3,\xi)$ is a contact 3-manifold, possibly non-compact, with co-oriented contact structure $\xi=\ker\alpha.$
\begin{thm}
\label{thm:localnbhd}
Let $K \subset (M^3,\xi=\ker \alpha)$ be a smooth co-oriented knot inside a contact three-manifold with a co-oriented contact structure $\xi=\ker\alpha$, and choose a parametrization $\gamma(\theta)\in K$. There exists a neighborhood $U \supset K$ that admits a contact embedding
$$ \phi \colon (U,\xi) \hookrightarrow (J^1S^1,\xi_{st}=\ker(dz-pd\theta)) $$
that extends the map
$$\gamma(\theta) \mapsto (\theta,p,z)=(\theta,-\alpha(\dot\gamma(\theta)),0)$$
where the value of the $p$-coordinate measures the failure of the Legendrian property.
\end{thm}
\begin{proof}
  We start by choosing a contact form $\alpha$ on $M$. Then we pick a generic smooth family $P_\theta \subset T_{\gamma(\theta)}M$ of tangent two-plane fields along $K$ that are transverse to both the line field $TK$ and the contact planes $\xi$ (the latter condition just means that the plane does not coincide with $\xi$); in particular, the intersection $P_\theta \cap \xi_{\gamma(\theta)}$ is one-dimensional. One can e.g.~start by choosing a generic family of two-planes that are transverse to $TK$. Then we choose a pair of smooth non-vanishing vector fields $V_1$, $V_2$ of the rank-2 vector bundle $P \to K$, where $V_1 \in P \cap \xi$. Note that $\xi$ is orientable along $K$ since the contact-structure is co-orientable, while $P$ is orientable along $K$ since the knot is co-orientable; hence $V_1$ is a trivial real line-bundle. We then choose $V_2$ so that $\langle V_1(\theta),V_2(\theta) \rangle=P_\theta$ form a basis at every point. The condition that $K$ is co-orientable is used in the last step.  
After renormalizing, we may require that $\alpha(V_2)=1$ is satisfied.

Using these two vector fields and the exponential map, we can construct a smooth embedding
$$ \psi \colon U \hookrightarrow J^1S^1 $$
of a neighborhood $U \supset K$ that extends the map
$$\gamma(\theta) \mapsto \{(\theta,p,z)=(\theta,-\alpha(\dot\gamma(\theta)),0)\}$$
and whose differential maps the vector field $V_1$ to $\partial_p$ and $V_2$ to $\partial_z$. 
It follows that $\psi$ pulls back $\alpha_{st}=dz-p\,d\theta$ to a contact form
$$\beta \coloneqq \psi^*\alpha_{st}$$
that satisfies $\beta|_{T_{\gamma(\theta)}M}=\alpha|_{T_{\gamma(\theta)}M}$ along the knot $K$.

Since the contact manifold $M$ is three-dimensional and $\ker \alpha=\ker \beta$ along $K$, the convex interpolation $\beta_t=(1-t)\beta + t\alpha$ is a family of contact forms along $K$. Since being a contact form is an open condition, $\beta_t$ are all contact forms in some small neighborhood of $K$.

A standard application of Moser's trick, see e.g.~the proof of \cite[Theorem 2.5.22]{Geiges:Intro}, produces a smooth isotopy $\psi_t$ with $\psi_0=\id_M$ defined in some small neighborhood of $K$, where $\psi_t|_K=\id_K$ and $(\psi \circ \psi_t)^*\alpha_{st}=e^{F_t}\beta_t$ for some $F_t \colon M \to \R$.
In other words, $\psi \circ \psi_1$ is the sought contact embedding.
\end{proof}

\begin{cor}
  \label{cor:isoext}
Consider a smooth isotopy
  $$\gamma_t \colon A \hookrightarrow (M^3,\xi=\ker \alpha)$$
 of a union of knots and arcs $A$ that is fixed near the boundary $\partial A,$ and which satisfies $\gamma_t^*\alpha=\eta_t^*(\gamma_0^*\alpha)$ for some smooth path of reparametrizations $\eta_t \colon A \to A$, $\eta_0=\id_A$, that fixes a neighborhood of the boundary. 
 Then the path of embeddings $\gamma_t \circ \eta_t^{-1}$ is induced by an ambient contact isotopy that can be taken to fix a neighborhood of the boundary.
\end{cor}
\begin{proof}
The pull-back of $\alpha$ is constant under the path of embeddings $\gamma_t \circ \eta_t^{-1}$. The proof of Theorem \ref{thm:localnbhd} can be extended to produce a smooth family of contact embeddings $\psi_t \colon U_t \hookrightarrow J^1A$ of neighborhoods $U_t \supset \gamma_t \circ \eta_t^{-1}(A)$, where the images $\psi_t(\gamma_t \circ \eta_t^{-1}(A))$ moreover remain fixed in the family. In addition we may assume that this family of embeddings is fixed near the boundary of $A$.

Considering the inverses $\psi_t^{-1}$, we obtain a family of contact embeddings whose domain is fixed and contains $\psi_0(\gamma_0).$ Since contact isotopies are generated by Hamiltonians, there exists a global contact isotopy $\varphi_t$ of $M$ for which $\psi_t^{-1}=\varphi_t \circ \psi_0^{-1}$. In particular,
$$\varphi_t \circ \gamma_0=\gamma_t \circ \eta_t^{-1}$$
holds as sought.
  \end{proof}

  \begin{lem}
\label{lem:1}
    Let $K$ be a non-Legendrian knot inside a contact manifold $(M^3,\xi)$. In any neighborhood of $K$ there exists a non-Legendrian knot $K_1$ which can be identified with
  $$\{ p=g(\theta), z=0\} \subset (J^1S^1=S^1_\theta \times \R_p \times \R_z,\ker (dz-p\,d\theta))$$
  under a locally defined contactomorphism, where $g^{-1}(0) \subsetneq S^1$ is a \underline{finite} union of closed intervals.
  \end{lem}
\begin{proof}
  According to Theorem \ref{thm:localnbhd}, there exists a contact embedding of a neighborhood $U \supset K$ inside $M$ into an open subset of
  $$(J^1S^1=S^1_\theta \times \R_p \times \R_z,\ker(dz-p\,d\theta)),$$
  under which $K$ is identified with a curve of the form $C=\{p=f(\theta), z=0\}$ and $U$ is identified with a neighborhood $U_C \supset C$ in $J^1S^1$. Since $K$ is non-Legendrian by assumption, the function $f$ is not everywhere zero.

  One can find a finite number of pairwise disjoint neighborhoods of the form
  $$O_{r_i,[a_i,b_i]} \coloneqq \{\theta \in [a_i,b_i],\:z^2+p^2\le r_i^2 \} \subset U_C, \:\:i=1,\ldots,N,$$
  where we have used the identification $S^1=\R/2\pi\Z$, such that
  $$C \setminus \bigcup_{i=1}^N O_{r_i,[a_i,b_i]}$$
  consists of a finite number of transverse arcs. (Note that the transverse part of $C$ is equal to $C_{tr}=C \setminus \{p=0\}$.) We can moreover assume that $C$ intersects each $\partial O_{r_i,[a_i,b_i]}$ transversely in the boundary stratum $\{z^2+p^2=r_i^2\} \subset \partial O_{r_i,[a_i,b_i]}$.

  Consider a family $f_t(\theta)$ of smooth functions for which $f_0=f$ and such that $C_t \coloneqq \{p=f_t(\theta),z=0\}$ coincides with $C$ outside of $\bigcup_{i=1}^N O_{r_i,[a_i,b_i]}$, while $f_t(\theta)=e^{-\rho(t)}f(\theta)$ for $\theta  \in [a_i,b_i]$ where:
  \begin{itemize}
  \item $\rho(t) \ge 0;$
  \item $\rho(t)=0$ holds in a neighborhood of $\{a_i,b_i\};$ and
  \item $\rho(t)=t$ holds in the subset $f^{-1}(0) \subset \cup [a_i,b_i]$ (i.e.~the non-transverse part of $C$).
  \end{itemize}
 Corollary \ref{cor:isoext} can now readily be applied to produce the corresponding ambient contact isotopy that squeezes $C$ onto some knot $K_1= \{p=g(\theta),z=0\}$ for which $g^{-1}(0) \subsetneq S^1$ consists of a finite number of closed intervals.
\end{proof}

\begin{lem}
  \label{lem:2}
  Let $K_1$ be a non-Legendrian knot that is contactomorphic to
  $$\{z=0,p=g(\theta)\} \subset (J^1S^1=S^1_\theta \times \R_p \times \R_z,dz-p\,d\theta)$$
  where $g^{-1}(0) \subsetneq S^1$ is a finite union of closed intervals. Then there exists a contact isotopy that squeezes $K_1$ onto a knot $K_2$ that satisfies the following.
  \begin{itemize}
    \item $K_2$ is contained in an arbitrarily small neighborhood of $K_1.$  
\item $K_2$ is nowhere negatively transverse (for some choice of orientation).
\item The non-transverse part of $K_2$ again consists of a finite union of closed intervals.
  \end{itemize}
\end{lem}
\begin{proof}
We will construct a contact isotopy that fixes the subset $\{g > 0\}$, while the remaining parts are squeezed onto Legendrian arcs.

  Any component of $\{g < 0\}$ is a (negative) transverse arc, and thus has a standard neighborhood of the form
  $$  \left(B^2_\epsilon \times \R_z, \ker\left(dz-(1/2)\left(y\,dx-x\,dy\right)\right)\right),$$
  where in polar coordinates on $B^2$ we can express the contact form as
  $$dz-(1/2)\left(y\,dx-x\,dy\right)=dz-\frac{r^2}{2}\,d\theta$$
  We can squeeze these transverse arcs onto the Legendrian arcs $s\mapsto (z=s,\theta=2s/r_0^2,r=r_0)$ described in polar coordinates, by the transverse isotopy
  $$s\mapsto (z=s,\theta=2s/r_0^2,r=t\cdot r_0), \:\: t \in [0,1).$$
  Here we use Corollary \ref{cor:isoext} to produce the ambient isotopy.

  After an interpolation, and for $r_0 >0$ sufficiently small, we can produce a contact isotopy that squeezes $K_1$ onto a knot which is negatively transverse only inside a finite number of arbitrarily small Darboux balls
  $$B^3_\epsilon \subset (\R^3,\ker(dz-ydx)),$$
  where these Darboux balls can be assumed to be contained inside an arbitrarily small neighborhood of $K_1$. Furthermore, for an appropriate interpolation, the knot can be assumed to intersect the Darboux ball transversely in a Legendrian unknotted tangle that intersects the boundary of the ball transversely in precisely two points. 
 
Finally, it is possible to construct a contact isotopy that squeezes these tangles onto a Legendrian tangle while fixing the boundary of the Darboux ball. Again we allude to Corollary \ref{cor:isoext} in order to produce the ambient contact isotopy.
\end{proof}

\begin{lem}
  \label{lem:3}
  Inside any neighborhood of a non-Legendrian knot $K_2$ that is positively transverse except at finite number of Legendrian arcs (i.e.~satisfies the conclusion of the previous lemma) there exists a neighborhood that is contactomorphic to
  $$ U \subset \left(\RR^2_{(x,y)} \times S^1_\theta, \ker\left(d\theta-(1/2)\left(y\,dx-x\,dy\right)\right)\right)$$
  where moreover
  \begin{itemize}
  \item $(\{(0,0)\} \times S^1) \cup B^3_\epsilon(\{((0,0),0)\}) \cup \ldots \cup B^3_\epsilon(\{((0,0),N)\}) \subset U$; and
    \item the contactomorphism takes $K_2$ to a knot that coincides with $\{(0,0)\} \times S^1$ outside of the balls $B^3_\epsilon(\{((0,0),i)\})$, while its image inside each of these balls is a smoothly unknotted arc with two boundary points contained in the boundary of the ball.
    \end{itemize}
\end{lem}
\begin{proof}
Use Theorem \ref{thm:localnbhd} to map $K_2$ to the graph
$$\{z=0,p=g(\theta)\} \subset (J^1S^1=S^1_\theta \times \R_p \times \R_z,\ker(dz-p\,d\theta))$$
under a contactomorphic embedding of a neighborhood $U \supset K_2$. Recall that a smooth family of knots of the form
$$\{z=0,p=g_t(\theta)\} \subset (J^1S^1=S^1_\theta \times \R_p \times \R_z,\ker(dz-p\,d\theta)),$$
for which the $g_t(\theta)$ differ by pre-compositions with isotopies of $S^1,$ can be realized by an ambient contact isotopy by Corollary \ref{cor:isoext}.

After a suitable such isotopy, supported in an arbitrarily small neighborhood of $K_2$, we may assume that $g_1(\theta) \ge 0$ has the property that it vanishes precisely inside a finite number of intervals $[a_i,a_i+\epsilon]$ where $\epsilon>0$ is arbitrarily small. For $\epsilon>0$ sufficiently small we are guaranteed the existence of round Darboux balls centered at $(\theta,p,z)=(a_i,0,0)$ of radius $r=2\epsilon$ that are entirely contained inside $U$. Obviously these Darboux balls cover the non-transverse part of the knot.

The part of the knot outside of these Darboux balls is positively transverse. One can connect these arcs by positively transverse arcs inside the Darboux balls to form a closed transverse knot $K_{tr} \subset U$. The sought neighborhood is finally given by the union consisting of a suitable standard neighborhood of $K_{tr}$ together with the previously constructed Darboux balls.
  \end{proof}

  \begin{proof}[Proof of Theorem \ref{thm:squeezing}]
    In view of Lemmas \ref{lem:1}, \ref{lem:2}, and \ref{lem:3}, it suffices to produce a contact isotopy of a knot
    $$K \subset (\{(0,0)\} \times S^1) \: \cup \: B^3_\epsilon(\{((0,0),0)\}) \: \cup \: \ldots \: \cup \: B^3_\epsilon(\{((0,0),N)\})$$
that squeezes it onto the transverse knot $\{(0,0)\} \times S^1$, where we can assume that $K \cap B^3_\epsilon(\{((0,0),i)\}$ is an unknotted arc, and where $K$ coincides with the transverse knot $\{(0,0)\} \times S^1$ outside of these balls.

The contact isotopy can be taken to fix the arcs
$$K \setminus \left(B^3_\epsilon(\{((0,0),0)\}) \cup \ldots \cup B^3_\epsilon(\{((0,0),N)\})\right),$$
while, inside each Darboux ball, it acts on $K$ by the rescaling
    $$(x,y,z) \mapsto (e^{-t}x,e^{-t}y,e^{-2t}z).$$
(Here we consider a Darboux ball centered at the origin.) Corollary \ref{cor:isoext} is used in order to ensure that this isotopy is induced by an ambient contact isotopy.
  \end{proof}

\section{Smooth $C^0$-limits of Legendrians are Legendrian (proof of Theorem \ref{thm:c0})}
 
The statement that the image is a Legendrian is an immediate consequence of the following lemma together with Corollary \ref{cor:nonsqueezing}.

\begin{lem}
  \label{lem:squeezing}
  Under the assumptions of the theorem, the Legendrian $\Lambda$ is squeezed onto $K$ by the sequence $\varphi_i \colon M \to M$ of contactomorphisms. (See Definition \ref{defn:squeezing}.)
\end{lem}
\begin{proof}
For $i_0 \gg 0$, we may assume that the contactomorphisms $\varphi_i \circ \varphi_{i_0}^{-1}$ with $i \ge i_0$ all are arbitrarily close in $C^0$-distance to the identity. 
  
Part (1) of the definition: We need to show that for $i_0 \gg 0$, there exists a tubular neighborhood of $K$ that contains the image of $\varphi_i(\Lambda)$ for all $i \ge i_0$, in which the latter is smoothly isotopic to $K$. This follows from Lemma \ref{lem:smoothiso} below.

Part (2) of the definition follows immediately from the $C^0$-convergence.

\end{proof}

\begin{lem}
\label{lem:smoothiso}
Consider a smooth knot $K \subset M$ and a fixed tubular neighborhood $\mathcal{N} \supset K$. Let $\phi \colon M \to M$ be a smooth map which is sufficiently $C^0$-close to a homeomorphism $\psi$ that satisfies $\psi(\Lambda)=K$. Then we may assume that $\phi(\Lambda) \subset \mathcal{N}$ is smoothly isotopic to $K$ inside of $\mathcal{N}$.
\end{lem}
\begin{proof}

It suffices to show that $\pi_1(\mathcal{N} \setminus \phi(\Lambda))=\Z^2$ since the existence of a smooth isotopy inside $\mathcal{N}$ from $\phi(\Lambda)$ to $K$ is then a consequence of the classical fact that the Hopf link is detected by the fundamental group of its complement; see \cite{Neuwirth}. To that end, note that $\mathcal{N}$ is a solid torus, i.e.~the complement of an unknot in $S^3$.

  Consider nested closed tubular neighborhoods
  $$K \subset \mathcal{N}_1 \subsetneq \mathcal{N}_2 \subsetneq \mathcal{N}$$
  that hence satisfy the property that the inclusion $\partial\mathcal{N}_2 \subset \mathcal{N} \setminus \mathcal{N}_1$ is a homotopy equivalence between a torus and a fattened torus.
  
  We consider the tubular neighborhood $\mathcal{N}_\Lambda \coloneqq \psi^{-1}(\mathcal{N}_2)$ of $\Lambda$. For $\phi$ sufficiently $C^0$-close to $\psi$, we may assume that $\phi(\partial\mathcal{N}_\Lambda) \subset \mathcal{N} \setminus \mathcal{N}_1$ is satisfied. Since the map is a $C^0$-approximation of $\psi$, it is clearly homotopic to $\psi$. Hence, it follows that
$$(\phi|_{\partial\mathcal{N}_\Lambda})_* =(\psi|_{\partial\mathcal{N}_\Lambda})_* \colon \pi_1(\partial\mathcal{N}_\Lambda) \to \pi_1(\mathcal{N} \setminus \mathcal{N}_1)$$
is an isomorphism of fundamental groups. In other words, the inclusion $\phi(\partial\mathcal{N}_\Lambda) \subset \mathcal{N} \setminus \mathcal{N}_1$ also induces an isomorphism of fundamental groups
$$(\iota_{\phi(\partial \mathcal{N}_\Lambda)})_* \colon \pi_1(\phi(\partial \mathcal{N}_\Lambda)) \to \pi_1(\mathcal{N} \setminus \mathcal{N}_1).$$

First we claim that the rank of $\pi_1(\mathcal{N} \setminus \phi(\Lambda))$ is at least equal to two. This follows since the previously established isomorphism
$$(\iota_{\phi(\partial \mathcal{N}_\Lambda)})_* \colon \pi_1(\phi(\partial \mathcal{N}_\Lambda)) \cong \Z^2 \to \pi_1(\mathcal{N} \setminus \mathcal{N}_1)$$
of groups factors through $\pi_1(\mathcal{N} \setminus \phi(\mathcal{N}_\Lambda))$. (Recall that $\phi(\partial\mathcal{N}_\Lambda) \subset \mathcal{N} \setminus \mathcal{N}_1$ and that there is a homeomorphism $\mathcal{N} \setminus \phi(\mathcal{N}_\Lambda)\cong \mathcal{N} \setminus \phi(\Lambda)$ since $\mathcal{N}_\Lambda \supset \Lambda$ is a tubular neighborhood.)

Second, we claim that the inclusion
$$\iota \colon \mathcal{N} \setminus \mathcal{N}_1 \hookrightarrow \mathcal{N} \setminus \phi(\Lambda) $$
induces a surjection
$$ \iota_* \colon \pi_1(\mathcal{N} \setminus \mathcal{N}_1) \cong \Z^2 \to \pi_1(\mathcal{N} \setminus \phi(\Lambda))$$
of fundamental groups. Namely, since the inclusion $\mathcal{N} \setminus \phi(\mathcal{N}_\Lambda) \subset \mathcal{N} \setminus \phi(\Lambda)$ is a deformation retract, considering the composition of inclusions
$$ \mathcal{N} \setminus \phi(\mathcal{N}_\Lambda) \subset \mathcal{N} \setminus \mathcal{N}_1 \subset \mathcal{N} \setminus \phi(\Lambda)$$
we see that $\iota_*$ factorizes through an isomorphism
$$ \pi_1(\mathcal{N} \setminus \phi(\mathcal{N}_\Lambda)) \cong \pi_1(\mathcal{N} \setminus \phi(\Lambda)).$$

Finally, the fact that the surjective group homomorphism $\iota_* \colon \Z^2 \to \pi_1(\mathcal{N} \setminus \phi(\Lambda))$ in addition is injective now follows by purely algebraic considerations, using the previously established fact that the rank
 of $\pi_1(\mathcal{N} \setminus \phi(\Lambda))$ is at least equal to two. (The rank of $\Z^2$ is equal to two and that any quotient of $\Z^2$ by a non-trivial subgroup has rank strictly less than two).

\end{proof}

Now that we know $K$ is Legendrian, it remains to show that $K$ is the contactomorphic image of $\Lambda$. We establish this by showing that $\phi_i(\Lambda)$ is Legendrian isotopic to $K$ for $i \gg 0$. We may assume that $\phi_i(\Lambda) \subset J^1K$ is contained inside a standard contact neighborhood of $K$, and that $\phi_i(\Lambda)$ moreover is \emph{smoothly} isotopic to $j^10=K$ inside the same neighborhood. 
Moreover, we prove the following.

\begin{prop}
\label{prop:Classify}
The Legendrian knot  $\varphi_i(\Lambda) \subset J^1K \subset M$ for $i \gg 0$ has the same classical invariants as the Legendrian $K=j^10$ (rotation number, Thurston--Bennequin invariant, smooth isotopy class) when considered inside the standard contact neighborhood $J^1K$ of the Legendrian knot $K=\phi_\infty(\Lambda)$.
\end{prop}
\begin{rem}In the case when the contact manifold $(M,\xi)$ satisfies $H_1(M)=H_2(M)=0$ and the absolute Thurston--Bennequin invariant thus is well-defined, the same ideas as the proof of Proposition \ref{prop:Classify} can be used to show something stronger: if the Legendrian $\Lambda_0$ can be squeezed onto the Legendrian $\Lambda_1$, and $\Lambda_1$ can be squeezed onto $\Lambda_0$, then $\Lambda_0$ and $\Lambda_1$ are contactomorphic.
  \end{rem}

\begin{proof}
  If we take $i_0 \gg 0$ sufficiently large, then $\varphi_{i}(\Lambda) \hookrightarrow U_K \subset (M,\xi)$ for all $i \ge i_0$, where $U_K \hookrightarrow J^1K$ is contactomorphic to standard contact neighborhood of $K$ in which the latter is identified with $j^10$. Furthermore, we may assume that $\varphi_i \circ \varphi_{i_0}^{-1}$ is $\epsilon$-close to the identity on some neighborhood $B_{r}(K) \subset J^1K$, while $\varphi_{i_0}(\Lambda) \subset B_{r/2}(K)$, for some fixed $r >0$ and $\epsilon>0$ arbitrarily small.

First we show that for any knot $K' \subset B_r(K) \setminus B_{r/2}(K)$ in the same homology class as $K,$ the relative Thurston--Bennequin numbers ${\tt tb}_{K'}(\varphi_{i_0}(\Lambda))$ and ${\tt tb}_{K'}(K)$ as computed inside $U_K,$ are the same.
 Since Lemma \ref{lem:tbineq} implies $ {\tt tb}_{K'}(\varphi_{i_0}(\Lambda)) \le {\tt tb}_{K'}(K),$ it suffices to prove 
 $ {\tt tb}_{K'}(\varphi_{i_0}(\Lambda)) \ge {\tt tb}_{K'}(K).$

Consider the sequence $(\varphi_i \circ \varphi_{i_0}^{-1})^{-1}$ of inverses of the above contactomorphisms, which $C^0$-converges to $\varphi_{i_0} \circ \varphi_\infty^{-1}$. Lemma \ref{lem:squeezing} implies this sequence of contactomorphisms squeezes the Legendrian $K$ onto $\varphi_{i_0}(\Lambda)$. As before, we again assume the contactomorphisms to be $\epsilon$-close to the identity on $B_r(K)$ for $i \ge i_0$.

  For $i \gg 0$ the squeezing property implies that
  $$(\varphi_i \circ \varphi_{i_0}^{-1})^{-1}(K) \subset U_\Lambda$$
  where $U_\Lambda \hookrightarrow J^1\Lambda$ is a standard contact neighborhood of $\varphi_{i_0}(\Lambda)$ in which the latter is identified with $j^10$. Again  Lemma \ref{lem:tbineq} implies
    $${\tt tb}_{K'}(\varphi_{i_0}(\Lambda)) \ge {\tt tb}_{K'}((\varphi_i \circ \varphi_{i_0}^{-1})^{-1}(K))={\tt tb}_{\varphi_i \circ \varphi_{i_0}(K')}(K).$$
  Since $\varphi_i \circ \varphi_{i_0}|_{K'}$ is $\epsilon$-close to the identity, we may assume that $\varphi_i \circ \varphi_{i_0}(K')$ is homologous to $K'$ inside $B_{r}(K) \setminus B_{r/2}(K)$. This immediately implies that ${\tt tb}_{\varphi_i \circ \varphi_{i_0}(K')}(K)={\tt tb}_{K'}(K)$ is satisfied.

  This finishes the proof of the equality of relative Thurston--Bennequin numbers
  $$ {\tt tb}_{K'}(\varphi_{i_0}(\Lambda)) = {\tt tb}_{K'}(K).$$
  in $U_K \hookrightarrow J^1K$.

Since the smooth isotopy types are clearly the same, it remains to establish an equality between rotation numbers. 
Recall the Thurston--Bennequin inequality 
$$ {\tt tb}(F(\varphi_{i_0}(\Lambda))) +|{\tt rot}(F(\varphi_{i_0}(\Lambda)))| \le -1,$$
where $F \colon J^1S^1 \hookrightarrow (\R^3,\xi_{st})$ is a contact embedding that takes $j^10$ to the standard unknot $\Lambda_{st}$ \cite{Bennequin:Entrelacements}. Then
$${\tt tb}(F(\varphi_{i_0}(\Lambda)))={\tt tb}(F(j^10))={\tt tb}(\Lambda_{st})=-1,$$
 implies the vanishing of the rotation number.

\end{proof}

Proposition \ref{prop:Classify} combined with the classification result of Legendrian knots inside $J^1S^1$ in \cite{DingGeiges} due to Ding--Geiges produces the sought-after Legendrian isotopy from $\varphi_i(\Lambda)$ to $K$ for $i \gg 0$ confined inside the standard contact neighborhood $J^1K \subset M.$


\bibliographystyle{alpha}
\bibliography{references}

\end{document}